\documentclass[12pt]{article}

\usepackage{amssymb}
\usepackage{amsfonts}
\usepackage{amsthm}
\usepackage{amsmath}
\usepackage{float}
\usepackage{bigints}
\usepackage{fullpage}
\usepackage{hyperref}

\newtheorem{thm}{Theorem}[section]
\newtheorem{cor}{Corollary}[section]
\newtheorem{lemma}{Lemma}[section]
\newtheorem{sublemma}{Sublemma}[section]
\theoremstyle{definition}
\newtheorem{rem}{Remark}[section]
\newcommand{\SD}{\mathcal{S}}
\newcommand{\FF}{\mathcal{F}}
\newcommand{\ED}{\mathbb{E}^d}
\newcommand{\QD}{\mathbb{Q}^d}
\renewcommand{\S}{\mathbb{S}}
\renewcommand{\Re}{\mathbb{R}}
\newcommand{\Ze}{\mathbb{Z}}
\newcommand{\Qe}{\mathbb{Q}}
\newcommand{\B}{\mathbf{B}}
\newcommand{\spi}[2][]{[#2]^{#1}_{\mathrm{s}}}
\newcommand{\sconv}[2][]{\mathrm{conv^{#1}_{s}}\left(#2\right)}
\renewcommand{\star}[1]{\mathrm{st}\left(#1\right)}
\newcommand{\stsp}[1]{\mathrm{st_{s}}\left(#1\right)}
\newcommand{\kers}{\ker_{\mathrm{s}}}
\newcommand{\st}{\; : \; }
\newcommand{\remark}[1]{}

\DeclareMathOperator{\bd}{bd}
\DeclareMathOperator{\cl}{cl}
\DeclareMathOperator{\crad}{cr}

\DeclareMathOperator{\inter}{int}
\DeclareMathOperator{\conv}{conv}
\DeclareMathOperator{\card}{card}

\usepackage{graphicx}

\title{Spindle Starshaped Sets 
}

\author{K\'{a}roly Bezdek\thanks{Partially supported by a Natural Sciences and 
Engineering Research Council of Canada Discovery Grant.} and M\'arton Nasz\'odi 
\thanks{Partially supported by the Hung. Nat. Sci. Found. (OTKA) grants: K72537 
and PD104744 and by the J\'anos Bolyai Research Scholarship of the Hungarian 
Academy of Sciences.}}

\begin{document}

\maketitle

\begin{abstract}
In this paper, spindle starshaped sets are introduced and investigated, which 
apart from normalization form an everywhere dense subfamily within the family 
of starshaped sets. We focus on proving spindle starshaped analogues of recent  
theorems of Bobylev, Breen, Toranzos, and Zamfirescu on starshaped sets. 
Finally, we consider the problem of guarding treasures 
in an art gallery (in the traditional linear way as well as 
via spindles).
\end{abstract}

\section{Introduction}

We denote the origin of Euclidean $d$-space $\ED$ by $o$, a closed Euclidean
ball in $\ED$ centered at $z$ of radius $\lambda$ by $\B[z,\lambda]$, its
boundary, the sphere by $\S(z,\lambda)$. When $\lambda$ is omitted, it is 1.
For the \emph{circumradius} of a set $A\subseteq\ED$, we use $\crad
A=\inf\{r>0\st
A\subseteq\B[q,r] \mbox{ for some } q\in\ED\}$.

For any $\lambda>0$ we define the \emph{$\lambda-$spindle} of two points
$x,y\in\ED$ as 
\[
\spi[\lambda]{x,y}=\bigcap \{\B[u,\lambda]\st {u\in\ED};{x,y\in \B[u,\lambda]}
\} 
\]
if $|x-y|\leq 2\lambda$ (with $|\cdot|$ standing for the standard norm of 
$\ED$), and as $\spi[\lambda]{x,y}=\ED$ otherwise. Unless
otherwise noted, $\lambda=1$, and $\lambda$ is omitted. Clearly, for any
$\lambda <\nu$ we have $\spi[\nu]{x,y}\subset \spi[\lambda]{x,y}$.

We recall that for a set $A\subseteq\ED$ and a point $p\in A$, the
\emph{visibility
region} of $p$ in $A$ is 
\[\star{p, A}=\{q\in A\st [p,q]\subseteq A\},\]
where $[p,q]$ refers to the line segment joining $p$ and $q$. When obvious from 
the context, we may omit $A$. The \emph{kernel} of $A$ is
$\ker{A}=\{p\in A\st \star{p, A}=A\}$. We say that $A$ is a \emph{starshaped
set} if $\kers{A}\neq\emptyset$. In particular, a starshaped set is non-empty.

We define the spindle analogues of these notions. For a set $A\subseteq\ED$ and
a point $p\in A$, the \emph{spindle visibility
region} of $p$ in $A$ is 
\[\stsp{p, A}=\{q\in A\st \spi{p,q}\subseteq A\}.\]
When obvious from the context, we may omit $A$. The \emph{spindle kernel} of $A$
is $\kers{A}=\{p\in A\st \stsp{p, A}=A\}$. We say that $A$ is a \emph{spindle
starshaped set} if $\kers{A}\neq\emptyset$. In particular, a spindle starshaped
set is non-empty.

We recall from \cite{BLNP07} that $A$ is called \emph{spindle convex} if
$A=\kers{A}$. Note that if $\crad A>1$ (resp., $\crad A>2$), then $A$ is 
spindle convex (or spindle
starshaped) if, and only if, $A=\ED$. 
As we will see (Corollary~\ref{cor:kerintersection}),  if $S$ is a spindle 
starshaped set in $\ED$, then its spindle kernel $\kers(S)$ is spindle convex.

A motivation for the study of spindle starshaped sets is that a star
shaped set with $C^2$ boundary whose curvature is bounded away from zero
is necessarily spindle starshaped with respect to $\lambda$-spindles
for a sufficiently large $\lambda$.

Krasnosselsky's well-known theorem \cite{ Kr46} states the following: Let $S$ 
be a compact set in $\ED$. Assume that for any $d+1$ points
$x_1,\ldots,x_{d+1}$ of $S$, there is a point $y\in S$ such that
$\mathop{\cup}\limits_{i=1}^{d+1} [y,x_i]\subset S$. Then $S$ is a starshaped 
set. Bobylev \cite{Bo99}, \cite{Bo01} observed that the same proof gives the 
following 
somewhat stronger result: Let $S$ be a compact set in $\ED$. Then the kernel of 
$S$ can be obtained as
\begin{equation}\label{eq:krasnowell}
 \ker S = \bigcap_{x\in S} \conv \big(\mathrm{st}(x,S)\big),
\end{equation} 
where $\conv (\cdot )$ stands for the convex hull of the corresponding set.

We prove the following (somewhat stronger) spindle analogue of 
(\ref{eq:krasnowell}).

\begin{thm}\label{thm:krasnowellspindle}
 Let $S$ be a compact set in $\ED$. Then the spindle kernel of $S$ can be
obtained as
$$
 \kers S = \bigcap_{x\in \bd S} \sconv{\stsp{x,S}},
$$ 
where $\sconv \cdot $ denotes the spindle convex hull of the given set, i.e., 
the intersection of all spindle convex sets containing the given set and $\bd 
S$ 
stands for the boundary of $S$ in $\ED$.
\end{thm}

This theorem combined with Helly's theorem \cite{He23,DGK63} yields

\begin{cor}\label{thm:krasno}
 Let $S$ be a compact set in $\ED$. Assume that for any $d+1$ points
$x_1,\ldots,x_{d+1}$ of $\bd S$, there is a point $y\in S$ such that
$\mathop{\cup}\limits_{i=1}^{d+1}\spi{y,x_i}\subset S$. Then $S$ is a spindle 
starshaped set. 
\end{cor}

Theorem~\ref{thm:krasnowellspindle} combined with Klee's version of Helly's
theorem \cite{Kl53} yields the following

\begin{cor}\label{thm:QKspindle}
 Let $S$ and $K$ be compact sets in $\ED$. Assume that for any $d+1$ points
$x_1,\ldots,x_{d+1}$ of $\bd S$, there is a vector $u\in \ED$ such that
$u+K\subseteq \bigcap_{i=1}^{d+1} \sconv{\stsp{x_i,S}}$. Then there is a vector 
$v\in \ED$ such that
$v+K\subseteq \kers S$.
\end{cor}

The next theorem is a discrete relative of Theorem~\ref{thm:krasnowellspindle}. 
It is based on sets called {\it flowers} that have been introduced by Gordon 
and Meyer \cite{GoMe94} and have been studied by Csikos \cite{Cs01} as well 
from the point of view of the Kneser-Poulsen conjecture. Here, we need the 
following version of flowers. Let $F$ be a {\it lattice polynomial}, i.e., an 
expression built up from some variables using the binary operations $\cap$ and 
$\cup$ with properly placed brackets indicating the order of the evaluation of 
the operations. We identify two lattice polynomials if they can be obtained 
from one another using the commutativity and associativity of the operations 
$\cap$ and $\cup$. (Other lattice identities are not used in the 
identification.) Also, it will be convenient to write $F$ as $F(x_1, \dots , 
x_n)$ indicating the variables of $F$ by $x_1, \dots , x_n$. We shall assume 
that each variable $x_i$ occurs exactly once in $F$. A {\it flower-polyhedron} 
in $\ED$ is a set of the form $F( B_1, \dots , B_n)$, where $F$ is a lattice 
polynomial and the sets $B_1, \dots , B_n$ are closed unit balls in $\ED$. 
Finally, a flower-polyhedron in $\ED$ is called {\it reduced along its 
boundary} if the boundary of the flower-polyhedron intersects the boundary of 
each generating unit ball in a $(d-1)$-dimensional set.

\begin{thm}\label{thm:flowers}
Let $F( B_1, \dots , B_n)$ be a flower-polyhedron  reduced along its boundary 
in $\ED$. Then $\kers \left(F( B_1, \dots , B_n)\right) = \bigcap_{i=1}^{n}B_i 
$.
\end{thm}

Clearly, Theorem~\ref{thm:flowers} leads to a geometric construction of the 
family of flower-polyhedra with $d$-dimensional spindle kernels in $\ED$. 
The set $\SD$ of all compact spindle starshaped sets in $\ED$ equipped with the 
Hausdorff distance is a Baire space by Baire's Category Theorem, since it is
a complete metric space. Now, recall that a property is called \emph{typical} 
for a member of a Baire space, if the set of
those members that do not have the property is of category one, i.e., they are 
a union of countably many nowhere dense sets. Although it is easy to see that 
flower-polyhedra with $d$-dimensional spindle kernels in $\ED$ form an 
everywhere dense set in $\SD$, still the following theorem holds, which in 
fact, is a spindle analogue of a theorem of Zamfirescu \cite{Za88} claiming 
that the kernel of a typical starshaped set is a singleton. 

\begin{thm}\label{thm:zamfi}
 The spindle kernel of a typical compact spindle starshaped set of $\ED, d\geq 
2$ is a singleton.
\end{thm}

Next, we prove an analogue of a result of Toranzos and Forte Cunto 
\cite{FoTo97} that
characterizes locally kernel points of starshaped sets. Let $S\subset\ED$ be a 
compact set. A point $x$ of $S$ is a \emph{spindle peak}
of $S$ if there is a neighborhood $U$ of $x$ such that for any $x^\prime\in U$
we have $\stsp{x^\prime}\subseteq \stsp{x}$.

\begin{thm}\label{thm:spindleToCu}
 Let $S\subset\ED$ be a compact set such that $\cl(\inter S)=S$ and $\inter S$
is connected. Assume that $x\in S$ is a spindle peak of $S$. Then $x\in\kers S$
(and hence, $S$ is necessarily spindle starshaped).
\end{thm}

Recently Bobylev \cite{Bo99}, \cite{Bo01} proved an elegant version of Helly's 
theorem for starshaped sets: Let $S_i\;\; (i\in I)$ be a family of compact 
starshaped sets in $\ED$
with $\card I\geq d+1$. Assume that for every $i_1,\ldots,i_{d+1}\in I$ the 
intersection $\cap_{j=1}^{d+1} S_{i_j}$ is starshaped. Then the intersection
$\cap_{i\in I} S_i$ is starshaped. Breen \cite{Br03} has strengthened that 
result as follows:  Let $S_i\;\; (i\in I)$ be a family of compact starshaped 
sets with $\card I\geq
d+1$ and $K$ a compact set in $\ED$. Assume that for every
$i_1,\ldots,i_{d+1}\in I$ there is a vector $u\in \ED$ such that 
$u+K\subseteq \ker\left(\mathop{\cap}\limits_{j=1}^{d+1}  S_{i_j}\right)$.
Then there is a vector $v\in \ED$ such that $v+K\subseteq \ker 
\left(\mathop{\cap}\limits_{i\in I} S_i\right)$. We prove the following spindle 
analogue of these results.

\begin{thm}\label{thm:QBs}
 Let $S_i\;\; (i\in I)$ be a family of compact spindle starshaped sets with 
$\card I\geq d+1$ and $K$ a compact set in $\ED$. Assume that for every
$i_1,\ldots,i_{d+1}\in I$ there is a vector $u\in \ED$ such that
$u+K\subseteq \kers\left(\mathop{\cap}\limits_{j=1}^{d+1}  S_{i_j}\right)$. 
Then there is a vector $v\in \ED$ such that $$v+K\subseteq \kers 
\left(\mathop{\bigcap}\limits_{i\in I} S_i\right).$$
\end{thm}

In \cite{Br05}, Breen proves that if every countable subfamily of a family 
$\FF$ of
starshaped sets (not necessarily compact) has a (non-empty) starshaped
intersection, then the intersection of $\FF$ is also starshaped.
The analogous statement for spindle starshaped sets follows. The proof
in our setting --since spindle starshaped sets are ''fat''-- is somewhat simpler
than Breen's. The main idea is to study the trace of our set family on
$\QD$.

\begin{thm}\label{thm:spindlebreengeneral}
  Let $K$ be a set in $\ED$ and $\FF$ be a family of spindle starshaped sets in
$\ED$ with the property that  the intersection of any countable subfamily of
$\FF$ is a spindle starshaped set whose spindle kernel contains a translate of
$K$. Then $\kers \bigcap \FF$ also contains a translate of $K$.
\end{thm}

Finally, we consider the problem of guarding only certain points in a planar art 
gallery -- a question known as ``Treasures in an art gallery'' in computational 
geometry \cite{urrutia_art_2000}. 
We characterize the case when a single guard suffices. 
We prove both a linear and a spindle version of the result. To our knowledge, 
both have been unknown.

\begin{thm}\label{thm:artgal}
 Let $S$ be a compact, simply connected set in the plane and $A$ some finite 
subset of $S$. Assume that for any three points of $A$ there is a point in $S$ 
from which each one is visible within $S$. Then there is a point in $S$ that 
can see all points of $A$.
\end{thm}

\begin{thm}\label{thm:artgalspin}
 Let $S$ be a compact, simply connected set of diameter at most 2 in the plane 
and $A$ some finite subset of $S$. Assume that for any three points of $A$ 
there is a point in $S$ from which each one is visible within $S$ via a 
spindle. Then 
there is a point in $S$ that can see all points of $A$ via spindles.
\end{thm}

Neither the planarity nor the simple connectedness condition can be dropped, 
see Remark~\ref{rem:artgal}.

In the rest of the paper we prove the stated theorems.

\section{Krasnosselsky--type result: Proof of Theorem \ref{thm:krasnowellspindle}}

In what follows, an \emph{arc}\label{def:arc} (or \emph{circular arc}) is a
connected subset of a circle, which contains no pair of antipodal points of the
circle.

Clearly, $ \kers S \subseteq \bigcap_{x\in \bd S} \sconv{\stsp{x,S}}$. So, we 
are 
left to show that
$$
\bigcap_{x\in \bd S} \sconv{\stsp{x,S}}  \subseteq   \kers S.
$$ 
We may assume that 
\[\bigcap_{x\in \bd S} \sconv{\stsp{x, S}}\neq\emptyset,\]
as otherwise there is nothing to prove. We pick a point $y_0$ from this
intersection. Note that
\begin{equation}\label{eq:SinBy}
 S\subseteq \B[y_0, 2].
\end{equation}
Indeed, let $z$ be a point of $S$ that is furthest from $y_0$. Suppose for a
contradiction that $r:=|y_0-z|>2$. Then $S\subseteq \B[y_0,r]$, and hence
$\stsp{z}$ is in the closed unit ball that touches $\S[y_0,r]$ at $z$ from
inside. It follows that $\sconv{\stsp{z}}$ is in this unit ball, too, and thus
does not contain $y_0$, a contradiction.

In three steps we will show that $y_0$ is a spindle star center of $S$.

\emph{Step 1.}
We show that for any $x\in S$ we have $\spi[\sqrt{2}]{x,y_0}\subseteq S$.
In this step, we follow closely Krasnosselsky's proof from \cite{Kr46}.

Suppose, for a contradiction, that there is an $x_0\in S$ such that
$\spi[\sqrt{2}]{x_0,y_0}\not\subseteq S$.

Then, by the compactness of $S$, there is an arc $\gamma$ of radius greater than
$\sqrt{2}$ connecting $x_0$ with $y_0$ such that some point $x^\prime\in\gamma$
is not in $S$. For any two points $a,b\in\gamma$, we denote the closed part of
$\gamma$ between $a$ and $b$ by $\gamma[a,b]$ and the open arc by $\gamma(a,b)$.

Let $x_1$ be the point of $S\cap \gamma[x^\prime, x_0]$ that is closest to
$x^\prime$ (note that $x_1\in\bd S$ may or may not coincide with $x_0$). Let 
$x_2$
denote a point in $\gamma(x_1, x^\prime)$ which is very close to $x_1$, more
precisely such that $d(x_1, x_2)<d(x^\prime, S$).
Finally, let $x_3$ and $y_1\in \bd S$ denote a pair of points of the sets 
$\gamma[x_2,
x^\prime]$ and $S$, respectively, at which the distance of the two (compact)
sets is attained. Note that $x_3$ may coincide with $x_2$, but is certainly not
the same as $x^\prime$. See Figure~\ref{fig:krasno}.

Now, we denote by $B$ the unit ball passing through $y_1$ with outer normal
vector $x_3-y_1$ at $y_1$.

On the one hand, we claim that $y_0$ is not in $B$. Consider the angle of the
vector $x_3-y_1$ with the direction vector of the non-degenerate arc
$\gamma[x_3, y_0]$ (oriented from $x_3$ toward $y_0$). From the choice of $x_3$
and $y_1$ it follows that this angle is at most a right angle. Since
$d(x_0,y_0)\leq 2$ (by
\eqref{eq:SinBy}), and the radius of $\gamma$ is greater than
$\sqrt{2}$, thus $\gamma$ is shorter than a quarter circle. Applying
Lemma~\ref{lem:ballcircle}, we obtain that $y_0\notin B$.

\begin{lemma}\label{lem:ballcircle}
  Let $z\in\ED$ and $x\in\ED\setminus\B[z]$ be given points. Let $C$ be a circle
of radius at least one with an arc $\mu$ which is shorter than a quarter circle
and has end points $x$ and $y$, oriented from $x$ toward $y$. Assume that the
angle of the vector $x-z$ with $\mu$ is at most a right angle. Then $y\notin
\B[z]$.
\end{lemma}
\begin{proof}
  Since $\B[z]\cap C$ is an arc of $C$ (which is shorter than a
  semicircle), $C \setminus \B[z]$ is longer than a semicircle. From the
  assumption on the angle of $\mu$ and $x-z$, it follows that if we consider the
  same orientation of $C\setminus \B[z]$ as that of $\mu$ then $x$ precedes the
  midpoint of $C\setminus \B[z]$. Since $\mu$ is shorter than a quarter circle,
  $\mu\subset C\setminus \B[z]$.
\end{proof}

On the other hand, we claim that $y_0\in B$. Indeed, if a point $z\in \ED$ is
not
in $B$ then $\spi{y_1,z}$ has points closer to $x_3$ than $y_1$, and thus $z$
cannot be in $\stsp{y_1, S}$. That is, $\stsp{y_1, S}\subset B$, and hence,
$\sconv{\stsp{y_1, S}}\subset B$ from which it follows that $y_0\in B$.


\begin{figure}[tb]
\begin{minipage}[b]{0.48\columnwidth}%
    \centering
    \includegraphics[width=0.93\textwidth]{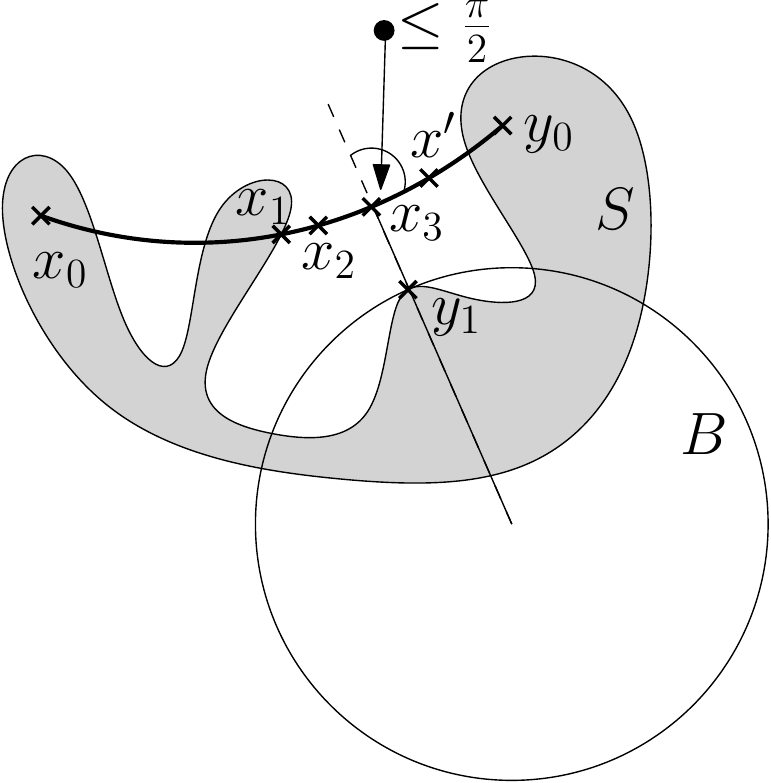}
    \caption[]{}
    \label{fig:krasno}
\end{minipage}%
\hfill%
\begin{minipage}[b]{0.52\columnwidth}%
\centering
    \includegraphics[width=0.93\textwidth]{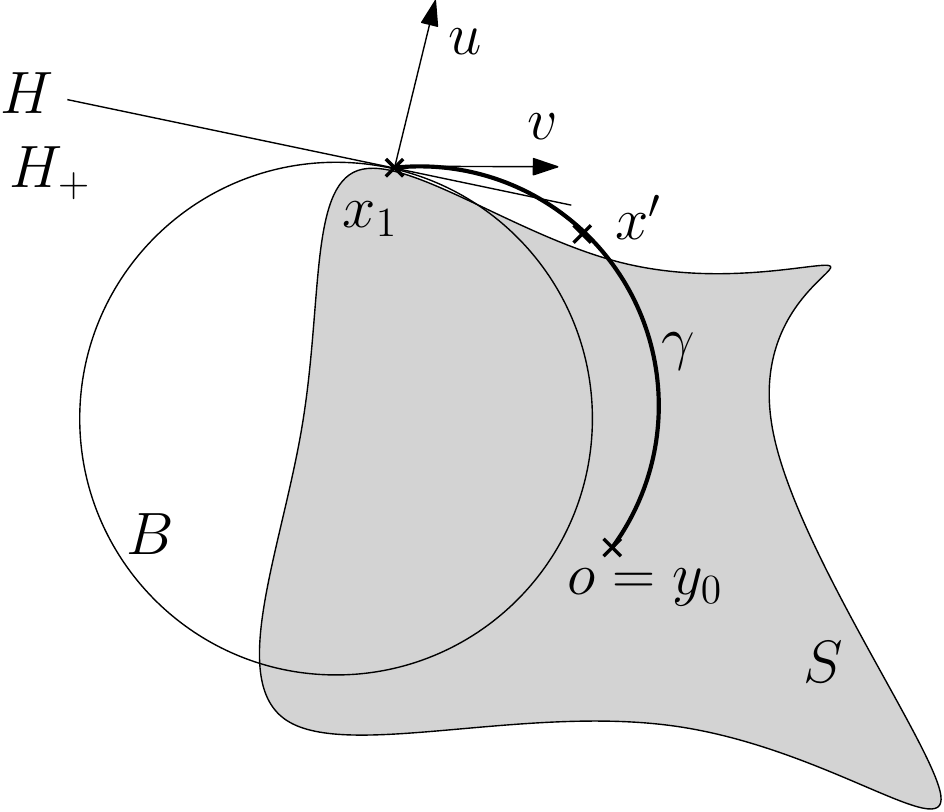}
    \caption[]{}
    \label{fig:krasnoo}
\end{minipage}
\end{figure}

\emph{Step 2.}
A case with a differentiability assumption.

Suppose, for a contradiction, that there is an $x_0\in
S\setminus\{y_0\}$ such that for some arc $\gamma$ of radius
$1< r < \sqrt{2}$ there is a point $x^\prime\in\gamma$ which is not in $S$.
Let $x_1$ be the point of $S\cap \gamma[x^\prime, x_0]$ that is closest to
$x^\prime$ (note that $x_1\in \bd S$ may or may not coincide with $x_0$). We 
may 
assume without loss of generality that $y_0$ is the origin.

Assume that $\bd S$ is differentiable at $x_1$. More precisely, let $f: \ED\to
\Re\cup\{\infty\}$ be the \emph{distance function} (or \emph{gauge function}) of
$S$, that is
$f(x)=\inf\{\lambda>0 \st x\in\lambda S\}$. 
By Step 1, there is a neighborhood of $x_1$, in which the values of $f$
are all real and not infinity.
We assume that $f$ is differentiable at $x_1$. Let
$u=\left(\frac{\partial f}{\partial x^1},\ldots, \frac{\partial f}{\partial
x^d}\right)(x_1)$ denote the gradient of $f$ at the point $x_1$. Since 
$f(x_1)=1>0$ and $f$ is \emph{positively homogeneous} (i.e. $f(\mu x)=\mu f(x)$
for
any $x\in \ED$ and $\mu>0$), we have that $u$ is a non-zero vector.

Let $H$ denote the hyperplane through $x_1$ with normal vector $u$,
and $B$ denote the unit ball touching $H$ at $x_1$ with outer normal
vector $u$ at $x_1$. Let $H_+$ denote the closed halfspace bounded by $H$ and
containing $B$. See Figure~\ref{fig:krasnoo}.

On the one hand, we claim that $y_0$ is not in $B$. The open arc
$\gamma(x_1,x^\prime)$ is disjoint from $S$. From the differentiability of $f$
at $x_1$, it follows that the angle of $u$ and the direction vector (denote it 
by
$v$) of $\gamma[x_1, y_0]$ at $x_1$ (we consider $\gamma$ oriented from $x_1$
toward $y_0$) is at most a right angle. More precisely, consider a
parametrization of $\gamma$ as follows: $\hat\gamma:[0,1] \to \ED$, where
$\hat\gamma(0)=x_1$, and such that $\left(\frac{d}{dt} \hat\gamma(t) |_{t=0+}
\right)\neq 0$ (here ``$0+$'' means right--sided derivative at $t=0$). Then $0
\leq \frac{d}{dt} f(\hat\gamma(t)) |_{t=0+}$ since otherwise there would be an
$\varepsilon>0$ such that
$f(\hat\gamma(\varepsilon))<f(\hat\gamma(0))=f(x_1)=1$. By Step 1, $S$ is
starshaped and thus, it follows that $\hat\gamma(\varepsilon)\in S$, a
contradiction. Next,

\begin{equation}\label{eq:diff}
  \frac{d}{dt} f(\hat\gamma(t)) |_{t=0+} = \left(\frac{d}{dt} \hat\gamma(t)
  |_{t=0+} \right)\cdot \left(\frac{\partial f}{\partial x^1},\ldots,
  \frac{\partial f}{\partial
  x^d}\right)(x_1)=v\cdot u.
\end{equation}

Since $\gamma$ is of radius
greater than one and $\gamma[x_1, y_0]$ is not longer than a semicircle, it
follows that $y_0$ is not in $B$.

On the other hand, we claim that $y_0\in B$. Let $z\in \ED$ be a point that is
not in $B$. Then $\spi{x_1,z}$ contains an arc (call it $\omega$) starting at
$x_1$ that leaves the halfspace $H_+$. In other words, the
angle of the direction vector $w$ of $\omega$ (oriented from $x_1$ toward $z$)
at $x_1$ and the vector $u$ is acute, that is $w\cdot u >0$. The same
computation as in \eqref{eq:diff}, shows that $\frac{d}{dt} f(\hat\omega(t))
|_{t=0+}>0$, where $\hat\omega$ is a parametrization of $\omega$ such that
$\hat\omega(0)=x_1$. Hence, there is an $\varepsilon>0$ such that
$f(\hat\omega(t))>f(\hat\omega(0))=f(x_1)=1$ for all $t\in(0,\varepsilon)$. In
other words, a non-degenerate open starting section of $\omega$ is disjoint from
$S$. Hence, $\spi{x_1,z}\not\subset S$ that is, $z\notin \stsp{x_1, S}$. Thus we
proved that $\stsp{x_1, S}\subset B$, and hence, $\sconv{\stsp{x_1, S}}\subset
B$.
From the definition of $y_0$, we obtain that $y_0\in B$.

\emph{Step 3.} The general case.

Let $\gamma, f$ and $x_1\in \bd S$ denote the same as in Step 2. By the 
compactness of
$\bd S$ it follows that there is a neighborhood $U$ of $x_1$ such that for any 
point
$x_2$ in $U\cap \bd S$, there is an arc $\xi$ (of radius greater than one)
connecting $x_2$ with $y_0$ such that a non-degenerated open section of $\xi$
starting at $x_2$ is disjoint from $S$. By Step 2, it is sufficient to find a
point $x_2$ in $U$ for which $f$ is differentiable at $x_2$.

From Step 1 it follows that there is a neighborhood $U_1\subset U$ of $x_1$ in
which the values of $f$ are all real (and not infinity) and on which $f$ is
a Lipschitz function. This is not hard to see, for a similar result on
(linearly) starshaped sets with a full--dimensional kernel (which is not
necessarily the case here), see \cite{To67}.

By Rademacher's theorem (cf. Section VII/23 of \cite{DiBe02}), $f$ is
differentiable almost everywhere in $U_1$. From
the positive homogeneity of $f$ it follows that if $f$ is differentiable at a
point $z\in U_1$ then $f$ is differentiable at any point on the ray emanating
from the origin and passing through $z$. Since points of the form
$z=\mu x_2$ (where $x_2\in U_1\cap \bd S$ and $\mu>0$) are of positive
Lebesgue measure in $U_1$, it follows that there is a point $x_2$ in $\bd S\cap
U_1$ where $f$ is differentiable. This concludes the proof of
Theorem~\ref{thm:krasnowellspindle}.

\section{Flowers: Proof of Theorem~\ref{thm:flowers}} 

First, we show that $ \kers \left(F( B_1, \dots , B_n)\right) \subseteq 
\bigcap_{i=1}^{n}B_i$. Indeed, let us assume in an indirect way that there 
exists a point $x\in \kers \left(F( B_1, \dots , B_n)\right)\setminus 
B_{i_{0}}$ for some $1\le i_0\le n$. As $F( B_1, \dots , B_n)$ is reduced along 
its boundary in $\ED$, therefore there exists a point $y\in \bd B_{i_0}$ and an 
open ball $\B (y, \epsilon )$ centered at $y$ with radius $\epsilon >0$ such 
that

\begin{equation}\label{flower-1}
\B (y, \epsilon )\cap     \bd B_{i_0}= \B (y, \epsilon )\cap   \bd F( B_1, 
\dots , B_n).
\end{equation}

As $y\in \bd B_{i_0}$ and $x\notin B_{i_0}$, therefore $\spi{y,x}$ contains an 
arc starting at $y$ that leaves $B_{i_0}$. Clearly, this together with 
(\ref{flower-1}) contradicts to  
$x\in \kers \left(F( B_1, \dots , B_n)\right)$.

Second, we show that $ \kers \left(F( B_1, \dots , B_n)\right) \supseteq 
\bigcap_{i=1}^{n}B_i$. It is convenient to prove it by induction on $n$. 
The claim is obvious for $n=1$. So, assume that it holds for any positive 
integer at most n-1 and write $F( B_1, \dots , B_n)=G(B_1,\dots ,B_m)\cap 
H(B_{m+1}, \dots ,
B_n)$ (resp., $F( B_1, \dots , B_n)=G(B_1,\dots ,B_m)\cup H(B_{m+1}, \dots , 
B_n)$) with $1\le m\le n-1$, where $G$ and $H$ are lattice polynomials of at 
most $n-1$ variables.
We note that the flower-polyhedra $G( B_1, \dots , B_m)$ and $H( B_{m+1}, \dots 
, B_n)$ are reduced along their boundaries in $\ED$. Thus, by the inductive 
assumption
\begin{equation}\label{flower-2}
\kers \left(G( B_1, \dots , B_m)\right) \supseteq \bigcap_{i=1}^{m}B_i, \ {\rm  
and}\  \kers \left(H( B_{m+1}, \dots , B_n)\right) \supseteq 
\bigcap_{i=m+1}^{n}B_i\ .
\end{equation}
Hence, (\ref{flower-2}) implies in a straightforward way that
$$
\kers \left(F( B_1, \dots , B_n)\right) \supseteq \kers \left(G( B_1, \dots , 
B_m)\right)\cap \kers \left(H( B_{m+1}, \dots , B_n)\right)\supseteq 
\bigcap_{i=1}^{n}B_i\ ,
$$
finishing the proof of Theorem~\ref{thm:flowers}.

\section{Typical sets: Proof of Theorem~\ref{thm:zamfi}}

Denote by $\SD_n\subset\SD$ the family of those compact spindle starshaped sets,
whose kernel contains a ball of radius $\frac{1}{n}$. It is sufficient to prove
that $\SD_n$ is nowhere dense in $\SD$.

Let $S\in\SD_n$ be a spindle starshaped set and $\varepsilon>0$. We will show
that in the $\varepsilon$--neighborhood of $S$, there is an $S^\prime\in\SD$
such that in some neighborhood of $S^\prime$ there is no element of
$\SD_n$.

Let $x\in\inter\kers(S)$ be a point. Take a line $\ell$ through $x$,
and denote the two endpoints of the non-degenerate line segment $\ell\cap S$ by
$u$ and $v$. Let $u^\prime$ and $v^\prime$ be points on $\ell$ close to $u$ and
$v$, respectively, but not on the line segment $[u,v]$. Now we attach two small
``spikes'' onto $S$: let $S^\prime=S\cup\spi{x,u^\prime}\cup\spi{x,v^\prime}$.
We have $\kers S^\prime=\{x\}$. Furthermore, clearly, there is a $\mu>0$ such
that in the $\mu$--neighborhood of $S^\prime$ in $\SD$, no set has a spindle
kernel, which contains a ball of radius $\frac{1}{n}$. This finishes the proof.

\section{Local characterization: Proof of Theorem~\ref{thm:spindleToCu}}

The proof is somewhat simpler than the one in \cite{FoTo97}. First, 
$\inter\left(\stsp{y}\right)$ is not empty for any $y\in\inter S$, and
$S=\cl(\inter S)$ hence, $M=\inter\left(\stsp{x}\right)$ is not empty. Suppose
for a contradiction that $x$ is not in the spindle kernel of $S$ and so, $\inter
S\neq M$. Since $\inter S$ is connected and $M$ is open and not empty, it 
follows that
there is a point, say $t$, in $\inter S\cap \bd M$. Let $\varepsilon>0$ be such
that $B=\B[t,\varepsilon]\subseteq\inter S$. We have that $S$ contains
$B\cup\spi{x,t}$. Clearly, for any $\delta>0$ we can find a point $x^\prime\in
\B[x,\delta]$ on $(x,t)$ and a $t^\prime$ close to $t$ on the ray emanating from
$t$ in the direction $\overrightarrow{xt}$ such that $\spi{x^\prime,
t^\prime}\subset B\cup\spi{x,t}$. Since $t\in\inter\spi{x^\prime,t^\prime}$, we
have that $\stsp{x}$ does not contain $\stsp{x^\prime}$ contradicting the
assumption that $x$ is a spindle peak.

\section{Carath\'eodory's theorem for spindle convex hull revisited}

Recall that Carath\'eodory's theorem states that the convex hull of a set $X 
\subset\ED$ is the union of simplices with vertices in $X$.
We prove the following spindle convex analogue.

\begin{lemma}\label{caraforspindle}
Let  $X$ be a set in $\ED$. If $y\in\bd (\sconv X )$, then
there exists a subset $\{x_1, \dots , x_d\}$ of $X$ such that 
$y\in\sconv{\{x_1, \dots , x_d\}} $. Moreover, if $y\in\inter (\sconv X )$, then
there exists a set $\{x_1, \dots , x_{d+1}\}\subseteq X$ such that 
$y\in\inter(\sconv{\{x_1, \dots , x_{d+1}\}}) $. 
\end{lemma} 

In \cite{BLNP07} (see Theorem 5.7) this statement is proved in the special case 
when $X$ is closed. Since we need it in the general case, we outline the proof 
here.

\begin{proof}[Sketch of the proof of Lemma~\ref{caraforspindle}]
If $y\in\inter (\sconv X )$, then the claim easily follows from the case when 
$X$ is closed, by a standard approximation argument. Now, assume that $y\in\bd 
(\sconv X )$. Clearly, $\crad X\leq 1$, as otherwise $\sconv X=\ED$. By 
Lemma~3.1 of \cite{BLNP07}, there is a point $p$ such that $\B[p]\supseteq 
\sconv{X}$ and $y\in\S(p)$. There are two cases. 

First, assume that $X\cap\S(p)$ is contained in an 
open hemisphere $C$ of $\S(p)$. Then clearly, $\sconv{X}\cap\S(p)$ is also 
contained in $C$ (otherwise $C\cup\B(p)$ would be a spindle convex set 
containing $X$ and not containing $y$, where $\B(p)$ denotes the open unit ball 
centered at $p$), and thus, so is $y$. Then $y$ must be 
in $\mathrm{Sconv}{X\cap\S(p)}$ and Carath\'eodory's theorem for the sphere 
yields the desired result, where $\mathrm{Sconv}$ denotes the spherical convex 
hull within
$\S(p)$.
Next, assume that $X\cap\S(p)$ is not contained in any open hemisphere of 
$\S(p)$. Then there are $d+1$ points in $X\cap\S(p)$ such that the convex hull 
of those $d+1$ points contains $p$. Clearly, for some $d$ of those points, we 
either have that their convex hull still contains $p$ or their spherical convex 
hull within $\S(p)$ contains $y$. In both cases, the statement follows.
\end{proof}

The following statement shows that the spindle convex hull may be built ``from
bottom up'' in the same way as the convex hull. On the other hand,
one can regard that statement as an extension of Lemma~\ref{caraforspindle}.

\begin{lemma}\label{lem:spindlecvxhull-general}
Let $X_1,\dots , X_n$ be spindle convex sets in $\ED , d\ge 2$ and let 
$m=\min{\{n, d+1\}}$. Then 
\[\sconv{X_1\cup\dots\cup X_n}=\bigcup_{1\le i_1<\dots <i_m\le 
n}\left(\bigcup_{x_{i_1}\in X_{i_1},\dots , x_{i_m}\in X_{i_m}} \sconv{ \{ 
x_{i_1},\dots , x_{i_m}  \}   }\right).\]
\end{lemma}

\begin{proof}[Proof of Lemma~\ref{lem:spindlecvxhull-general}]

First, we prove the following special case.

\begin{sublemma}\label{sublem:spindlecvxhull-special} 
Let $Y$ be a spindle convex set in $\ED$ and $z\in\ED$. Then 
  \[\sconv{Y\cup\{z\}}=\bigcup_{y\in Y} \spi{y,z}.\]
\end{sublemma}

For the proof of Sublemma~\ref{sublem:spindlecvxhull-special}, we quote 
Lemma~5.6 of \cite{BLNP07}.

\begin{lemma}\label{lem:sconv}
Let $A\subset \ED$ be a set with $\crad (A)<1$, and let $\B[q]$ be a
closed unit ball containing $A$. Then

\begin{tabular}{cl}
(i)& 
$A\cap\S(q)$ is contained in an open hemisphere of $\S(q)$ and\\
(ii)&
$\sconv{A}\cap\S(q)=\mathrm{Sconv}(A\cap\S(q))$
\end{tabular}

where $\mathrm{Sconv}$ denotes the spherical convex hull within
$\S(q)$. 
\end{lemma}

\begin{proof}[Proof of Sublemma~\ref{sublem:spindlecvxhull-special}]
We need to show that the right hand side is spindle convex, for which it is
sufficient to show the claim in the special case when
$Y$ is the spindle of two points, say $Y=\spi{y_1,y_2}$.
\remark{
Indeed, if $u_1, u_2\in Y^\prime$ then
$u_1\in\spi{y_1, z}$ and $u_2\in\spi{y_2, z}$ for some $y_1,y_2\in Y$ and, in
order to prove that $\spi{u_1, u_2}\subseteq Y^\prime$, it is sufficient to show
that $\spi{u_1, u_2}\subseteq \bigcup\limits_{y\in \spi{y_1,y_2}} \spi{y,z}$.
}

  First, assume that $\crad (\{y_1, y_2, z\})<1$.
Denote by $Y^\prime=\bigcup_{y\in Y} \spi{y,z}$ and
by $\hat Y:=\sconv{\{y_1, y_2, z\}}$.
Clearly, $Y^\prime\subseteq\hat Y$, we prove the reverse containment, for
which it is sufficient to show that $\bd \hat Y\subseteq Y^\prime$.

Let $w$ be a point of $\bd \hat Y$ distinct from
$y_1, y_2$ and $z$. By Corollary~3.4. of \cite{BLNP07}, there is a supporting
unit
sphere $\S(q)$ of $\hat Y$ through $w$. By Lemma~\ref{lem:sconv}, $\S(q)\cap
\hat Y=\mathrm{Sconv}(\{y_1, y_2, z\}\cap\S(q))$. On the other hand, clearly,
$\mathrm{Sconv}(\{y_1, y_2, z\}\cap\S(q))\subseteq Y^\prime$. Hence,
$w\in Y^\prime$ finishing the proof in the case when a ball of radius less than
one contains $Y\cup\{z\}$.

Assume now that $\crad (\{y_1, y_2, z\})=1$. We may assume that $\B[o]$ is the
only unit ball that contains $\{y_1, y_2, z\}$. Clearly, $y_1, y_2$ and $z$ lie
on a great circle of $\S(o)$. If $y_1=-y_2$, we are done.
Otherwise, the shorter great circular arc on $\S(o)$ connecting $y_1$ and $y_2$
contains $-z$ (or else, $y_1,y_2$ and $z$ would be on an open great semi-circle
contradicting the assumption that $\crad (\{y_1, y_2, z\})=1$). On the other
hand, this great circular arc is contained in $\spi{y_1,y_2}\cap\S(o)$. Thus,
the right hand side contains $\spi{z,-z}=\B[o]$.

Finally, assume that $\crad (\{y_1, y_2, z\})=\lambda>1$. Then by the previous
paragraph,
\[\sconv[\lambda]{Y\cup\{z\}}=\bigcup_{y\in Y} \spi[\lambda]{y,z}\]
both being a ball of radius $\lambda$. Clearly, replacing $\lambda$ by 1 we
obtain a larger set on both sides, which may only be $\ED$.
\end{proof}

Before continuing with the proof of Lemma~\ref{lem:spindlecvxhull-general} we 
derive the following
spindle starshaped analogue of a theorem of Smith \cite{Sm68} on kernels of 
starshaped sets. In what follows, 
\emph{maximality} of a set is taken with respect to containment.

\begin{cor}\label{cor:kerintersection}
 Let $S$ be a spindle starshaped set in $\ED$. Then
\[
 \kers(S)=\cap\{ Y\subset S \st  Y \mbox{ is a maximal spindle convex subset of
} S\}.
\]
\end{cor}
\begin{proof}
 Since singletons are spindle convex sets and, by Zorn's lemma, every spindle
convex subset of $S$ is contained in a maximal spindle convex subset of $S$, the
left hand side contains the right one. The reverse containment follows from
Sublemma~\ref{sublem:spindlecvxhull-special}.
\end{proof}

Now, we prove Lemma~\ref{lem:spindlecvxhull-general} for $n=2$.

\begin{sublemma}\label{sublem:spindlecvxhull} 
Let $Y$, $Z$ be spindle convex sets in $\ED$. Then 
  \[\sconv{Y\cup Z}=\bigcup_{y\in Y , z\in Z} \spi{y,z}.\]
\end{sublemma}

\begin{proof}[Proof of Sublemma~\ref{sublem:spindlecvxhull}]
One can follow the setup of the proof of 
Sublemma~\ref{sublem:spindlecvxhull-special} and derive
the desired claim from the analogue spherical convexity claim in spherical 
$3$-space, which in fact, follows from the analogue
convexity claim in Euclidean $3$-space. The relevant somewhat laborious, but 
straightforward details we leave to the reader.
\end{proof}

Finally, we are ready to prove Lemma~\ref{lem:spindlecvxhull-general} by 
induction on $n$. The details are as follows. Recall that 
Sublemma~\ref{sublem:spindlecvxhull}
proves  Lemma~\ref{lem:spindlecvxhull-general} for $n=2$. Thus, we can assume 
that  Lemma~\ref{lem:spindlecvxhull-general} holds for all $n$ at most $k$ with 
$k\ge 2$, and then
we prove that it holds for $n=k+1$ as well. Indeed, if $m=d+1$,
then Lemma~\ref{lem:spindlecvxhull-general} simply follows from 
Lemma~\ref{caraforspindle} in a straightforward way. So, we may assume that 
$m=n=k+1<d+1$. Clearly, by induction 
$$
\sconv{X_1\cup\dots\cup X_k}=\bigcup_{x_{1}\in X_{1},\dots , x_{k}\in X_{k}} 
\sconv{ \{ x_{1},\dots , x_{k}  \}   }.
$$
Therefore it follows from Sublemma~\ref{sublem:spindlecvxhull} in a 
straightforward way that
$$
\sconv{X_1\cup\dots\cup X_{k}\cup 
X_{k+1}}=\sconv{(\left(\sconv{X_1\cup\dots\cup X_{k}}\right)  \cup X_{k+1}}
$$
$$
=\sconv{\left(   \bigcup_{x_{1}\in X_{1},\dots , x_{k}\in X_{k}} \sconv{ \{ 
x_{1},\dots , x_{k}  \}   }\right) \cup X_{k+1}}
$$
$$
=\bigcup_{x_{1}\in X_{1},\dots , x_{k}\in X_{k}, x_{k+1}\in X_{k+1}} \sconv{ \{ 
x_{1},\dots , x_{k}, x_{k+1}  \}   },
$$
finishing the proof of Lemma~\ref{lem:spindlecvxhull-general}.
\end{proof}

\section{Klee--type result: Proof of Theorem~\ref{thm:QBs}}

Some of the ideas of the proof come from \cite{Bo01}. The details are as 
follows.

We recall the notion of \emph{Minkowski difference} of two sets $A$ and $B$ in
$\ED$ (cf. p.133 of \cite{Sch93}):
\[
 A\sim B=\{x\in\ED\st x+ B\subseteq A\}.
\]

Let $S_\ast=\mathop{\cap}\limits_{i\in I} S_i$. By 
Corollary~\ref{thm:QKspindle}, it is
sufficient to show that for any $x_1,\ldots,x_{d+1}$ in $S_\ast$, we have a
translate of $K$ in $\cap_{t=1}^{d+1}\stsp{x_t, S_\ast}$.

For each $i\in I$, let 
\[T_i=\{x\in S_i \st \spi{x, x_t}\subseteq
S_i\;\;\forall 1\leq t\leq d+1\}.\]

Denote by $T_\ast=\mathop{\cap}\limits_{i\in I} T_i$. We need to prove that
$T_\ast\sim K\neq\emptyset$. Clearly,
\begin{equation}
 T_\ast\sim K=\mathop{\cap}\limits_{i\in I} (T_i\sim K).
\end{equation}
Hence, by the topological Helly theorem (see for example, \cite{Bo01} and 
\cite{DGK63}), it suffices to prove that for any
choice $i_1,\ldots,i_{d+1}$ of (not necessarily distinct)
indices from $I$,
\begin{equation}
 T^0:=(\mathop{\cap}\limits_{l=1}^{d+1} T_{i_l})\sim K   
=\mathop{\cap}\limits_{l=1}^{d+1} (T_{i_l}\sim K)
\end{equation}
is spindle starshaped.

Let $S^0=(\mathop{\cap}\limits_{l=1}^{d+1} S_{i_l})\sim K   
=\mathop{\cap}\limits_{l=1}^{d+1} \left(S_{i_l}\sim K\right)$. By
the hypothesis, $S^0$ is a non-empty, compact, spindle starshaped set.
To finish the proof, we will show that $\kers S^0\subseteq \kers T^0$.

Let $x$ be a point of $\kers S^0$ and $y\in T^0$ arbitrary. We need to show that
$\spi{x,y}\subseteq T^0$. We fix an index $i_l$, say $i_1$. We need 
$\spi{x,y}\subseteq T_{i_1}\sim K$ or, equivalently, that for any 
$u\in\spi{x,y}+K$ we
have $\spi{u,x_t}\subseteq S_{i_1}$ for any $t=1,\ldots,d+1$.

On the one hand, since $y\in T^0$, we have that
\begin{equation}
C_t:=\sconv{\big((y+K)\cup\{x_t\}\big)} \subseteq 
\mathop{\cap}\limits_{l=1}^{d+1}
S_{i_l}
\end{equation}
holds for all $1\le t\le d+1$.
On the other hand, $x+K\subseteq \kers \left(\cap_{l=1}^{d+1} S_{i_l}\right)$. 

It follows that
\[
 \cup\{\spi{a,b}\st a\in x+K, b\in C_t\}\subseteq 
\mathop{\cap}\limits_{l=1}^{d+1}
S_{i_l}
\]
holds for all $1\le t\le d+1$. We may assume that $K$ is spindle convex, since 
the spindle kernel of a spindle starshaped set is
spindle convex. Thus, Lemma~\ref{lem:spindlecvxhull-general} (more exactly, 
Sublemma~\ref{sublem:spindlecvxhull}) implies that
\begin{equation}\label{eq:convcup}
  \cup\{\spi{a,b}\st a\in x+K, b\in C_t\}=\sconv{((x+K)\cup C_t)},
\end{equation}
which clearly contains $\spi{u,x_t}$, finishing the proof of 
Theorem~\ref{thm:QBs}.

\section{Countable intersections: Proof of Theorem~\ref{thm:spindlebreengeneral}}

First, we prove the following: Let $\FF$ be a family of spindle starshaped sets 
in $\ED$ with the property
that  the intersection of any countable subfamily of $\FF$ is a spindle
starshaped (hence, non-empty) set. Then $\bigcap \FF$ is also starshaped. 
In fact, this is that special case of Theorem~\ref{thm:spindlebreengeneral} 
where $K$ is a singleton. The details are as follows.

 First, we define a new family, $\FF_1$ as follows. We enumerate members of
$\QD$ and carry out the following inductive algorithm. For each $q\in\QD$ if
there is an $F$ in our set family such that $q\notin F$ then we intersect each
member of the set family by $F$ to obtain the next set family. At the end of
this algorithm, we obtain the family $\FF_1$ whose countable subfamilies clearly
have spindle starshaped intersections.

  Now, $\QD\cap F$ is the same set for all $F\in\FF_1$. If this set is empty or
a singleton, then each set in $\FF$ is a singleton (otherwise they would have a
non-empty interior and hence, contain a rational point). Clearly, these
singletons must be identical, and the theorem follows.

  So we assume that $\QD\cap F$ contains more than one point (for each
$F\in\FF$). For a set $A\subset\ED$ we define its rational spindle kernel as
\[
 \kers^{\Qe}A=\left\{p\in A\st \spi{p,a}\cap\QD\subseteq A \mbox{ for each }
a\in A\cap \QD \right\}.
\]
Note that $\kers^{\Qe}A$ may contain non-rational points.

By definition and the fact that $\QD\cap F$ is the same set for all
$F\in\FF_1$, the following hold:
\begin{equation}\label{eq:kersqcontains}
\kers^{\Qe}F\supseteq \kers F\neq\emptyset \mbox{ for any } F\in\FF_1,
\end{equation}
and
\begin{equation}\label{eq:kersqcap}
 \kers^{\Qe} (F_1\cap F_2)=\kers^{\Qe} F_1 \cap \kers^{\Qe} F_2
 \mbox{ for any } F_1,F_2\in\FF_1.
\end{equation}

Assume that $\kers^{\Qe}F_1$ is a singleton, say $\{p\}$ for some $F_1\in\FF_1$.
Then by \eqref{eq:kersqcontains} and \eqref{eq:kersqcap}, $p$ is in $\kers
\bigcap\FF_1$. 

Thus, we may assume that for all $F\in\FF_1$, its rational spindle kernel
$\kers^{\Qe}F$ is not a singleton. Using 
Lemma~\ref{lem:spindlecvxhull-general}, we have
that 
\begin{equation}\label{eq:spindleinkers}
 \mbox{if } p,q\in \kers^{\Qe}F \mbox{ then } \inter \spi{p,q}\cap\QD \subseteq
\left(\kers^{\Qe}F\right)\cap \QD.
\end{equation}
Note that $\left(\kers^{\Qe}F\right)\cap \QD$ is the same set for all
$F\in\FF_1$.
It follows from \eqref{eq:spindleinkers} that the interior of
$\left(\kers^{\Qe}F\right)\cap \QD$ relative to
$\QD$ is not empty. Using the fact that $F$ is spindle starshaped, it is not
difficult to see that for any interior (relative to $\QD$) point $p$ of
$\left(\kers^{\Qe}F\right)\cap \QD$, we have that $p\in\kers F$. Thus for any
such $p$, we have $p\in\kers \left(\cap\FF_1\right)=\kers\left(\cap\FF\right)$
finishing the proof.


To prove the general case (ie., when $K$ is not a singleton),  we may follow 
the above proof up to
\eqref{eq:kersqcontains}, from which 
\begin{equation*}
\left(\kers^{\Qe}F\right)\sim K\supseteq \left(\kers F\right) \sim
K\neq\emptyset \mbox{ for any } F\in\FF_1,
\end{equation*}
follows. Similarly, from \eqref{eq:kersqcap} we obtain
\begin{equation*}
 \big(\kers^{\Qe} (F_1\cap F_2)\big)\sim K=\bigg[\big(\kers^{\Qe} F_1\big)\sim K
\bigg]\cap \bigg[\big(\kers^{\Qe} F_2\big)\sim K\bigg]
\end{equation*}
holds for any $F_1,F_2\in\FF_1$.
Again, we may assume that for all $F\in\FF_1$, $\kers^{\Qe}(F)\sim K$ is not a
singleton, otherwise the theorem follows easily. From \eqref{eq:spindleinkers} 
we
obtain
\begin{equation*}
 \mbox{if } p,q\in (\kers^{\Qe}F)\sim K \mbox{ then } \inter \spi{p,q}\cap\QD
\subseteq
\left((\kers^{\Qe}F)\sim K\right)\cap \QD.
\end{equation*}
Now, by the same inductive procedure that we used at the beginning of the proof
above, we may assume that $\left((\kers^{\Qe}F)\sim
K\right)\cap \QD$ is the same set for all $F\in\FF_1$. Finally, for any interior
(relative to $\QD$) point $p$ of $\left((\kers^{\Qe}F)\sim
K\right)\cap \QD$, we have that $p\in(\kers \cap \FF )\sim K$.

\section{Art gallery: Proofs of Theorems~\ref{thm:artgal} and 
\ref{thm:artgalspin}}\label{sec:artgal}

\begin{rem}\label{rem:artgal}
 The assumption of simple connectedness cannot be dropped. This is shown by the 
example of an annulus. For any $N\in\Ze^+$, if the inner circle is small enough 
then any $N$ points of the outer circle can be seen from some point of the 
annulus, but no point sees all the points (or a sufficiently large finite 
subset) of the outer circle.

This example can be turned into one in three--space: for any $N\in\Ze^+$ 
there is a homology cell $S$ in $\Re^3$ such that any $N$ points of a certain 
subset 
of $S$ can be seen from some point of $S$ but no point sees them all. We leave 
it as an exercise to the reader.
\end{rem}

We note that Remark~\ref{rem:artgal} applies to this spindle version of the 
theorem as well.

\begin{proof}[Proof of Theorem~\ref{thm:artgal}]
We call a set $F$ in $S$ \emph{geodesically convex} with respect to $S$ if for 
any $p,q\in F$ the shortest path in $S$ connecting $p$ and $q$ (which is unique 
by the simple connectedness of $S$) is contained in $F$. We claim the following:

\begin{enumerate}
 \item The intersection of geodesically convex sets is again geodesically
convex (w.r.t. $S$).
 \item A geodesically convex set is simply connected.
 \item $\star{x,S}$ is compact and geodesically convex for any $x\in S$.
\end{enumerate}

1. is obvious. 2. is easy to prove. Indeed, consider a subset $F$ of $S$ that 
is not simply connected. Then there is a line $\ell$ through some point of 
$F\setminus S$ whose intersection with $F$ is not connected. It clearly shows 
that $F$ is not geodesically convex.

To prove 3., let $p$ and $q$ be points of $\star{x,S}$, and consider the 
shortest path $\gamma$ connecting them within $S$. If $x,p$ and $q$ are 
collinear, then $\gamma$ is simply the line segment $[p,q]$, which is in 
$\star{x,S}$. If they are not collinear then the rays 
$\overrightarrow{xp},\overrightarrow{xq}$ bound two angular regions on the 
plane, one of which is convex, call it $T$. 
We may assume that $\gamma$ is not $[p,x]\cup[x,q]$, as otherwise we are done. 
Let $p^\prime$ (resp. $q^\prime$) be the point of 
$\gamma\cap\overrightarrow{xp}$ (resp. $\gamma\cap\overrightarrow{xq}$) closest 
to $x$. Consider the part $\gamma^\prime$ of $\gamma$ from $p^\prime$ to 
$q^\prime$. Clearly, $\gamma=[p,p^\prime]\cup\gamma^\prime\cup[q^\prime,q]$. 
Now, neither $p^\prime$ nor $q^\prime$ is $x$. Moreover, clearly, 
$\gamma^\prime\subset T$. By the simple connectedness of $S$, we have that 
$\gamma^\prime\subset \star{x,S}$ finishing the proof of the geodesic convexity 
of $\star{x,S}$. Its compactness follows from the compactness of $S$.

Finally, Theorem~\ref{thm:artgal} follows from these claims and the topological 
version of Helly's theorem.
\end{proof}

\begin{proof}[Proof of Theorem~\ref{thm:artgalspin}]
The only point where the proof of Theorem~\ref{thm:artgal} needs to be changed 
a bit is the proof of 3. First, we notice that by the simple connectedness of 
$S$, $\gamma^\prime$ lies in the triangle $\Delta=\Delta_{{p^\prime} x 
{q^\prime}}$. It is not difficult to see that for any point $y$ of this 
triangle, $\spi{x,y}\subseteq \Delta\cup \spi{x,{p^\prime}}\cup 
\spi{x,{q^\prime}}$. Now, again, the simple connectedness of $S$ yields that 
the each point of $\gamma^\prime$ is spindle visible from $x$.
\end{proof}

\medskip

\noindent
K\'aroly Bezdek
\newline
Department of Mathematics and Statistics, University of Calgary, Canada,
\newline
Department of Mathematics, University of Pannonia, Veszpr\'em, Hungary,
\newline
{\sf E-mail: bezdek@math.ucalgary.ca}
\newline
and
\newline
M\'arton Nasz\'odi
\newline
Institute of Mathematics, E\"otv\"os Lor\'and University, Budapest, Hungary,
\newline
{\sf E-mail: marton.naszodi@math.elte.hu}


\end{document}